\documentclass[11pt,usenames,dvipsnames,reqno]{amsart} 
\usepackage{xypic}
\usepackage{amsmath}
\usepackage{amssymb}
\usepackage{amsthm}

\usepackage[colorlinks=true,linkcolor=NavyBlue,citecolor=NavyBlue,anchorcolor=NavyBlue,urlcolor=NavyBlue]{hyperref}
\usepackage{graphicx}
\usepackage{enumitem}
\usepackage{tikz}
\usepackage{mdframed}

\usepackage{yfonts}
\usepackage[T1]{fontenc}

\newcommand{\cc}{\mathbb{C}}

\newtheoremstyle{thm}{15 pt}{10 pt}{\itshape}{}{\bfseries}{.}{.5em}{}
\newtheorem{thmx}{Theorem}

\newtheorem{cor-main}[thmx]{Corollary}
\newtheorem{theorem}{Theorem}
\numberwithin{theorem}{section}

\newtheorem{observation}[theorem]{Observation}
\newtheorem*{thm*}{Theorem}
\newtheorem*{cor*}{Corollary} 
\newtheorem*{lem*}{Lemma}

\newtheorem*{lemmaRealization*}{Realization Lemma}
\newtheorem*{problemA*}{Problem 1.100A}
\newtheorem*{problemB*}{Problem 1.100B}
\newtheorem*{problem*}{Problem 1.100A/B}

\newtheorem*{construction*}{Motivating construction}

\newtheoremstyle{ex}{10 pt}{10 pt}{\normalfont}{}{\bfseries}{.}{.5em}{}
\theoremstyle{ex}
\newtheorem{example}[theorem]{Example}

\newtheoremstyle{rem}{10 pt}{5 pt}{\normalfont}{}{\bfseries}{.}{.5em}{}
\theoremstyle{rem}
\newtheorem{remark}[theorem]{Remark}
\newtheorem*{rem*}{Remark}

\numberwithin{equation}{section}

\usepackage{color}

\begin{document}

\thispagestyle{empty}
\title[Unknotted ribbon disks and algebraic curves]{Cross-sections of unknotted ribbon disks and algebraic curves}
\author[K. Hayden]{Kyle Hayden} \address{Boston College, Chestnut Hill, MA 02467} \email{kyle.hayden@bc.edu}

\maketitle

\begin{abstract}
We resolve parts (A) and (B) of Problem 1.100 from Kirby's list by showing that many nontrivial links arise as cross-sections of unknotted holomorphic disks in the four-ball. The techniques can be used to produce unknotted ribbon surfaces with prescribed cross-sections, including unknotted Lagrangian disks with nontrivial cross-sections.
\end{abstract}

\section{Introduction} \label{sec:intro}

A properly embedded surface $\Sigma$ in $B^4$ is said to be \textbf{ribbon} if the restriction of the radial distance function from $B^4$ to $\Sigma$ is a Morse function with no local maxima. As an important class of examples, holomorphic curves in $B^4 \subset \cc^2$ are naturally ribbon.  This Morse-theoretic perspective has shed light on the topology of holomorphic curves $\Sigma \subset \cc^2$ and the links that arise as transverse intersections $L_r=\Sigma \cap S^3_r$, where $S^3_r$ is a sphere of radius $r>0$; see \cite{fiedler,bo:qp,borodzik:morse,hayden:stein}. Rudolph \cite{rudolph:qp-alg} and Boileau-Orevkov \cite{bo:qp} showed that the links obtained as cross-sections of holomorphic curves are precisely the \emph{quasipositive links}, a special class of braid closures. To fully exploit this characterization, it is necessary to understand the relationship between different cross-sections of the same holomorphic curve. This is the subject of Problem 1.100 in Kirby's list.

\begin{problem*}[\cite{kirby}]
 \label{prob:kirby} 
Let $\Sigma \subset \cc^2$ be a smooth algebraic curve, with $L_r=\Sigma \cap S^3_r$ defined as above, and let $g$ denote the Seifert genus.
\begin{enumerate}[label=\bf(\Alph*)]
\item If $r \leq R$, is $g(L_r) \leq g(L_R)$?
\item If $L_R$ is an unlink and $r \leq R$, is $L_r$ an unlink?
\end{enumerate}
\end{problem*}

For comparison, by Kronheimer and Mrowka's proof of the local Thom conjecture \cite{km:thom}, the slice genus $g_*$ is known to satisfy $g_*(L_r) \leq g_*(L_R)$ for $r \leq R$. In addition, since Seifert and slice genera are equal for strongly quasipositive links by \cite{rudolph:qp-obstruction}, the answer to part (A) of Problem~\hyperref[prob:kirby]{1.100} is ``yes'' when $L_r$ is strongly quasipositive.  Part (B) is a special case of part (A), and Gordon's work \cite{gordon:ribbon} on ribbon concordance implies that the answer to (B) is ``yes'' when $L_r$ is a knot, i.e.~connected. In spite of this,  we show the answers to (A) and (B) are both ``no''.

\begin{thmx}\label{thm:embed}
Every quasipositive slice knot arises as a link component in a cross-section of an unknotted holomorphic disk in $B^4 \subset \cc^2$.
\end{thmx}

The holomorphic disks constructed in the proof of Theorem~\ref{thm:embed} are compact pieces of algebraic curves, hence provide the desired counterexamples to part (B) of Problem~\hyperref[prob:kirby]{1.100}. The construction is inspired by a version of the Whitney trick that admits an appealing statement in terms of \textbf{ribbon-immersed} surfaces in $S^3$, i.e.~immersed surfaces with boundary whose singularities are of the type depicted in Figure~\hyperref[fig:ribbon-int]{\ref{fig:ribbon-int}(a)}: 
\emph{Every ribbon-immersed surface $F \subset S^3$ lies inside a larger one $F'$ that is isotopic (through immersions that are injective along the boundary) to an embedded surface $F''$.} Moreover, the surface $F'$ may be taken to have the same topological type as $F$; the construction is depicted in Figure~\ref{fig:ribbon-int}. The corresponding four-dimensional construction can be used to embed an arbitrary ribbon surface inside one that is unknotted, i.e.~isotopic to an embedded  surface in $\partial B^4$.

\begin{figure}\center
\def\svgwidth{.9\linewidth} 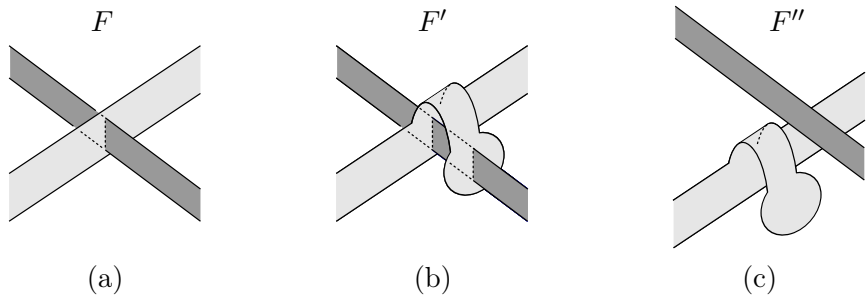 
\caption{A version of the Whitney trick that embeds one ribbon-immersed surface inside a simpler one.} 
\label{fig:ribbon-int}
\end{figure}

Any counterexamples to part (A) of Problem~\hyperref[prob:kirby]{1.100} constructed using  Theorem~\ref{thm:embed} will have disconnected cross-sections $L_r$. However, we can use an alternative construction to obtain a sharper result:

\begin{thmx}\label{thm:100a}
For every integer $n\geq 0$, there is a smooth algebraic curve $\Sigma_n$ and a pair of positive numbers $r <R$ such that $g(L_r)=2n+1$ and $g(L_R)=n+1$, where $L_r=\Sigma_n \cap S^3_r$ and $L_R=\Sigma_n \cap S^3_R$ are knots.
\end{thmx}

For each $\Sigma_n$, the smaller cross-section $L_r$ is a prime quasipositive 4-braid slice knot and $L_R$ is strongly quasipositive.

We also apply these techniques to study Lagrangian cobordisms between Legendrian links, as introduced in \cite{chantraine}. A \textbf{Lagrangian filling} of a Legendrian link $L$ in the standard contact $S^3$ is a properly embedded Lagrangian surface $\Sigma$ in the standard symplectic $B^4$ whose boundary is $L$. We further require the Lagrangian $\Sigma \subset B^4$ to be orientable, exact, and collared; for the significance of these latter two conditions, see \cite{chantraine:collar}.  As discussed in \cite{cornwell-ng-sivek:obstructions}, it is currently unknown if there are nontrivial knots that arise as collarable cross-sections of a Lagrangian filling of the unknot. We show that the answer is ``yes'' if the cross-section is allowed to be disconnected.

\begin{thmx}\label{thm:lagr}
There are infinitely many nontrivial links that arise as collared cross-sections of unknotted Lagrangian disks in the standard symplectic $B^4$.
\end{thmx}

The remainder of the paper is organized as follows. In \S\ref{sec:top}, we explain the topological construction underlying our primary results. In \S\ref{sec:100a}, we review the necessary background material on quasipositive braids and then prove Theorems~\ref{thm:embed} and \ref{thm:100a}. We close by constructing the Lagrangian disks needed to prove Theorem~\ref{thm:lagr} in \S\ref{sec:lagr}.

\begin{rem*}
In a previous version of this work, it was claimed that  Problem~\hyperref[prob:100b]{1.100B} has a positive answer. The argument relied on a statement about unknotted ribbon disks that was shown to be false by Jeffrey Meier and Alexander Zupan in a private correspondence \cite{meier-zupan:email}. Their  illumination of the error in the original argument  ultimately led to Theorem~\ref{thm:embed} above. I wish to thank  Meier and Zupan for their invaluable correspondence.
\end{rem*}

\subsection*{Acknowledgements} I wish to thank Lee Rudolph for bringing these questions to my attention and Eli Grigsby for numerous stimulating conversations about braids and ribbon surfaces. Thanks also to Peter Feller and Patrick Orson for comments on a draft, and to Siddhi Krishna, Clayton McDonald, Jeffrey Meier, and Alexander Zupan for helpful conversations.

\section{The topological construction}\label{sec:top}

This brief section offers an informal account of the  topological construction that inspires Theorem~\ref{thm:embed}. Given a link $L \subset S^3$, a \textbf{band} $b:I \times I \to S^3$ is an embedding of the unit square such that $L \cap b(I \times I)= b(I \times \partial I)$. A link $L'$ is obtained from $L$ by \textbf{band surgery} along $b$ if $L'=(L\setminus b(I \times \partial I) ) \cup b(\partial I \times I)$. For example, Figure~\hyperref[fig:stevedore]{\ref{fig:stevedore}(a-b)} presents the stevedore knot $6_1$ as the result of band surgery on an unlink. As illustrated in Figure~\hyperref[fig:stevedore]{\ref{fig:stevedore}(c)}, a description of a link $L$ as a result of band surgery on an unlink yields a natural ribbon-immersed spanning surface $F \subset S^3$ with $\partial F=L$: simply attach the bands to a collection of embedded disks bounded by the unlink.

\begin{figure}\center
\def\svgwidth{\linewidth} 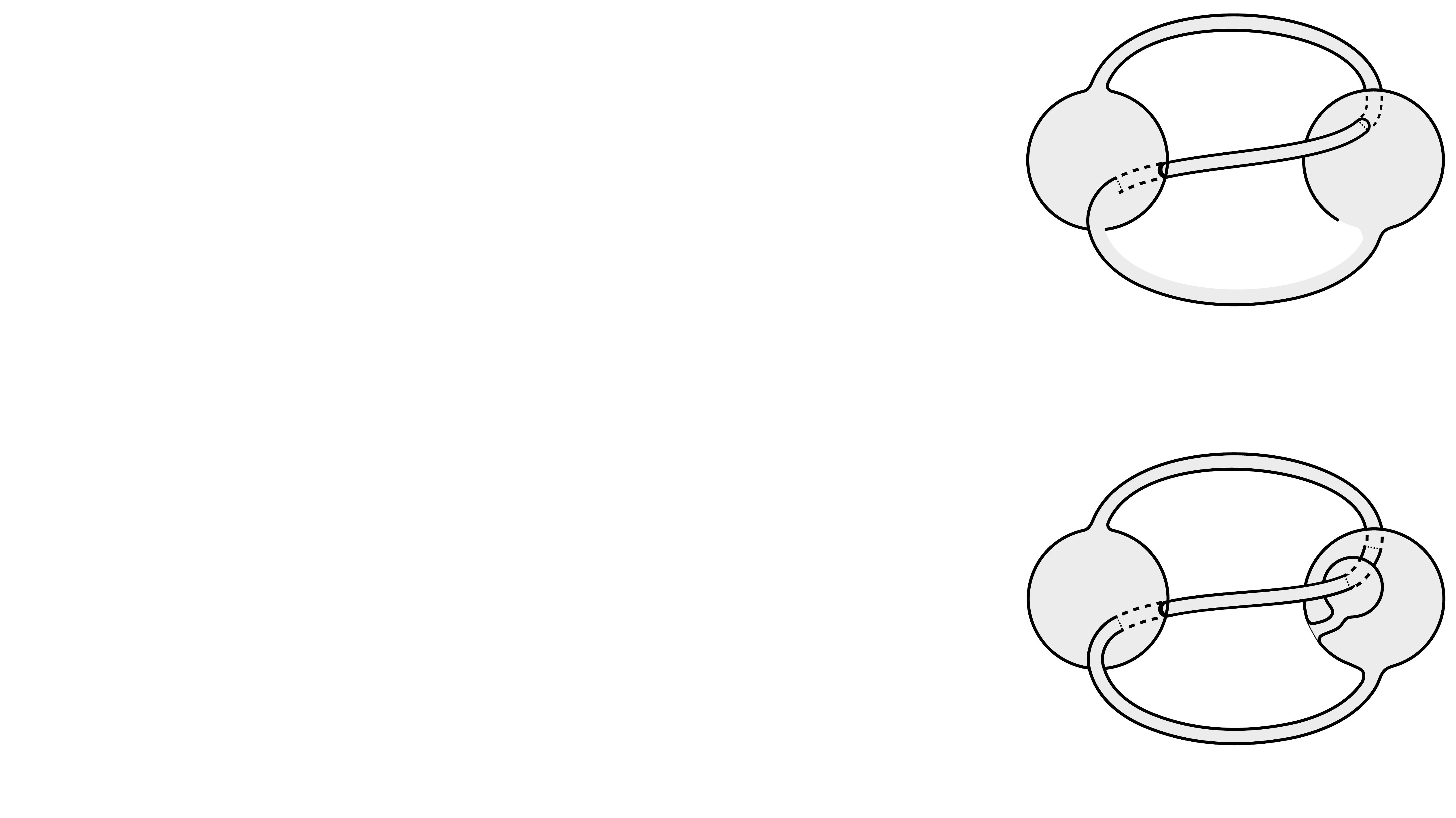 
\caption{Band surgeries on unlinks and the corresponding ribbon surfaces.} \label{fig:stevedore}
\end{figure}

In general, a ribbon-immersed surface $F\subset S^3$ will not be isotopic (through immersions that are injective along the boundary) to an embedded surface.  However, we can always find another ribbon-immersed surface $F' \subset S^3$ containing $F$ such that $F'$ is isotopic to an embedded surface.  The construction is summarized in Figure~\ref{fig:ribbon-int}. For an example, see Figure~\hyperref[fig:stevedore]{\ref{fig:stevedore}(f)}, where  the ribbon disk bounded by the stevedore knot from Figure~\hyperref[fig:stevedore]{\ref{fig:stevedore}(c)} is embedded in a ribbon disk for the unknot that is isotopic to a standard embedded disk. 

Under the correspondence between ribbon-immersed surfaces $F \subset S^3$ and ribbon surfaces $\Sigma \subset B^4$ (as in \cite{hass}), the preceding observation implies that every ribbon disk embeds as a connected component in a sublevel set of an unknotted ribbon disk. Restricting to the perspective of cross-sections, we obtain a topological version of Theorem~\ref{thm:embed}: 

\begin{observation}
Every knot that bounds a ribbon disk in $B^4$ arises as a link component in a cross-section of an unknotted ribbon disk in $B^4$.
\end{observation}

\section{Quasipositive band surgeries and algebraic curves}\label{sec:100a}

\subsection{Bands and Bennequin surfaces}

In \cite{bkl}, Birman, Ko, and Lee define an alternative presentation of the $n$-stranded braid group $B_n$ using generators of the form
\begin{equation*}
\sigma_{i,j}= (\sigma_i \cdots  \sigma_{j-2}) \sigma_{j-1} (\sigma_i \cdots \sigma_{j-2})^{-1}, \qquad 1 \leq i \leq j \leq n.
\end{equation*}
Such generators had been used earlier by Rudolph \cite{rudolph:braided-surface}, who termed $\sigma_{i,j}$ a (positive) \textbf{embedded band}.  
An expression of a braid in $B_n$ as a word $w$ in the generators $\sigma_{i,j}$ is called an \textbf{embedded bandword}. Given an embedded bandword $w$ for $\beta$, there is a canonical Seifert surface $F_w$, called the \textbf{Bennequin surface}, constructed from $n$ parallel disks by attaching a positively (resp.~negatively) twisted embedded band between the $i^{th}$ and $j^{th}$ disks for each term $\sigma_{i,j}$ (resp.~$\sigma_{i,j}^{-1}$) in $w$; see Example~\ref{ex:bennequin} below. Note that the Euler characteristic of the resulting surface is $n-|w|$, where $|w|$ denotes the length of the embedded bandword.

\begin{example}\label{ex:bennequin}

The Bennequin surfaces in Figure~\ref{fig:bennequin} correspond to embedded bandwords $\beta_0= \sigma_{1,4} \sigma_2 \sigma^{-1}_{1,3} \sigma_{2,4} \sigma_{1,3}$ and $\beta_1=\sigma_{1,4} \sigma_2 \sigma_1 \sigma_{1,3}^{-1} \sigma_1^{-1} \sigma_{2,4} \sigma_1 \sigma_{1,3} \sigma_1^{-1}$. These are minimal-genus Seifert surfaces for $\beta_0$ and $\beta_1$. Indeed, we can  confirm $g(\beta_0)=1$ and $g(\beta_1)=3$ using their Alexander polynomials (which have degree two and six, respectively) and the bound $g(K) \geq \deg(\Delta_K)/2$, where $K$ is any knot and $\deg(\Delta_K)$ is defined as the breadth of $\Delta_K$.

\begin{figure}\center
\def\svgwidth{\linewidth} 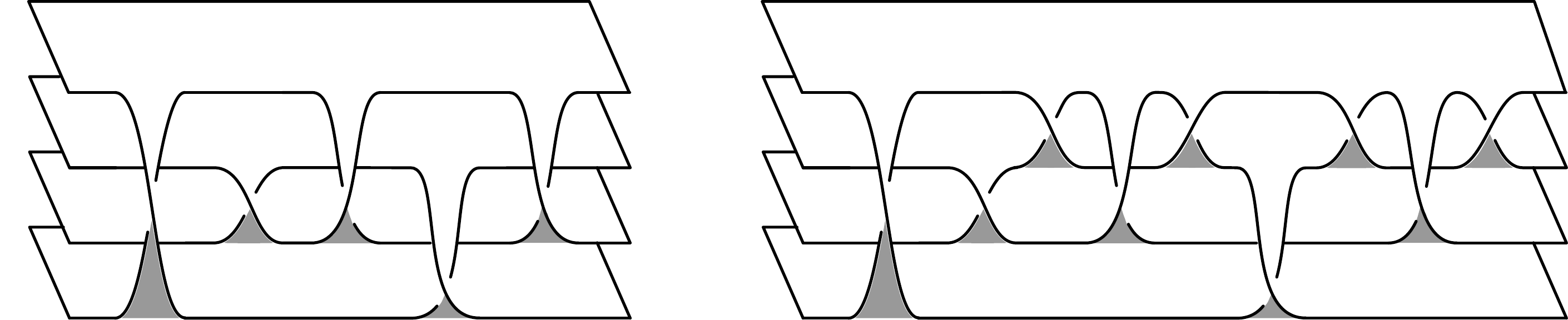 \caption{Bennequin surfaces for the braids $\beta_0$ (left) and $\beta_1$ (right).} \label{fig:bennequin}
\end{figure}

\end{example}

A braid is called \textbf{strongly quasipositive} if it is a product of positive embedded bands (\cite{rudolph:kauffman-bound}, cf.~\cite{rudolph:braided-surface}). More generally, we  define a braid to be \textbf{quasipositive} if it is a product of arbitrary conjugates of the standard positive generators, i.e.~a product of subwords $w \sigma_i w^{-1}$, which we simply call (positive) \textbf{bands}. An expression of a braid as a word in these bands is called a \textbf{bandword}, and a construction analogous to the one described above associates a ribbon-immersed spanning surface to each bandword.

\subsection{Construction of algebraic curves} We now build the desired algebraic curves using a special type of \emph{quasipositive band surgery} --- that is, the addition of a band $w \sigma_i w^{-1}$ to a quasipositive braid. This relies on the following lemma, due to Feller \cite[Lemma 6]{feller:optimal}.

\begin{lemmaRealization*}[Feller \cite{feller:optimal}, cf.~Orevkov \cite{orevkov:diagrams}, Rudolph \cite{rudolph:qp-alg}] 
Let $\beta$ and $\gamma$ be quasipositive $n$-braid words such that $\gamma$ can be obtained from $\beta$ by applying a finite number of of braid group relations, conjugations, and additions of a conjugate of a positive generator anywhere in the braid. Then there exists a polynomial of the form
$p(z,w)=w^n+c_{n-1}(z) w^{n-1}+\cdots+c_0(z)$
and constants $0<r<R$ such that the intersections of $\Sigma = \{p(z,w)=0\}$ with $S^3_r$ and $S^3_R$ are isotopic to $\beta$ and $\gamma$, respectively.
\end{lemmaRealization*}

To prove the main theorem, it suffices to adapt the topological construction from \S\ref{sec:top} using quasipositive braids and band surgeries.

\begin{proof}[Proof of Theorem~\ref{thm:embed}]
Let $\beta \in B_n$ be a quasipositive braid expressed as a product of positive bands $w \sigma_i w^{-1}$. Let $m$ denote the number of singularities  in the ribbon-immersed Bennequin surface $F$ for $\beta$. We begin by constructing a quasipositive braid $\beta' \in B_{n+m}$ that contains $\beta$ as a sublink.  Each ribbon singularity in $F$ corresponds to a region in the braid as depicted in Figure~\hyperref[fig:modification]{\ref{fig:modification}(a)} (or its rotation by 180$^\circ$, with strand orientations reversed); here the dashed strands represent the crossing strands in the positive band and may pass over or under the black strand in the center of the region where the dashed strands are not drawn. After braid isotopy of the crossing strands, we obtain the quasipositive braid in Figure~\hyperref[fig:modification]{\ref{fig:modification}(b)}. We then introduce an additional strand as shown in Figure~\hyperref[fig:modification]{\ref{fig:modification}(c)}; the new strand passes under all others except at the two points shown in the figure. After performing these modifications for each of the $m$ original  singularities, we obtain a quasipositive braid $\beta' \in B_{n+m}$ that contains $\beta$ as a sublink. Up to smooth isotopy, $\beta'$ is obtained from $\beta$ by adding $m$ unknots, linked with $\beta$ as depicted in Figure~\hyperref[fig:modification]{\ref{fig:modification}(d)}.

\begin{figure}\center
\def\svgwidth{\linewidth} 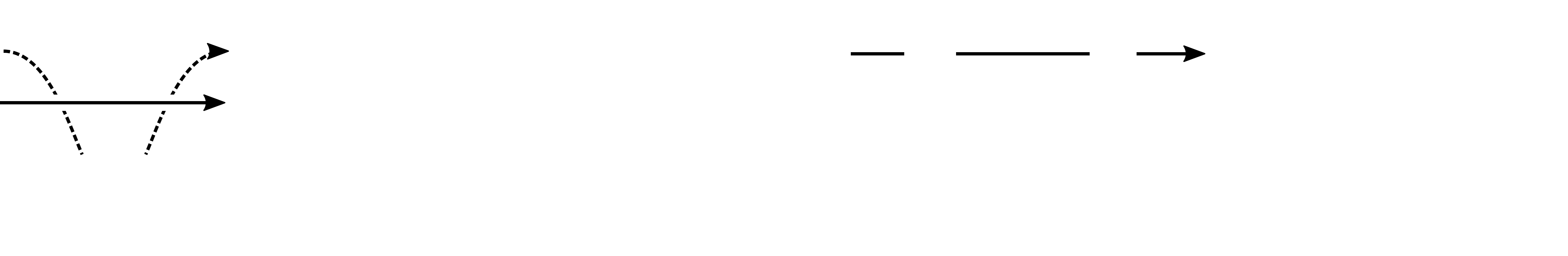 \caption{Modifying the original braid (a-c) and its closure (d).} \label{fig:modification}
\end{figure}

\begin{figure}[b]\center
\def\svgwidth{.915\linewidth} 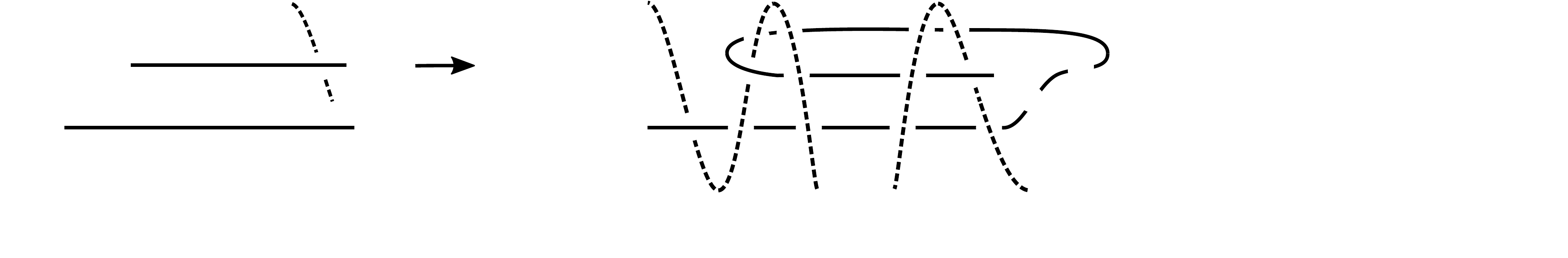 \caption{Band surgery and isotopy of the braid closure.} \label{fig:gamma}
\end{figure} 

Let $\gamma \in B_{n+m}$ be the braid obtained from $\beta'$ by adding a positive generator within each of the $m$ neighborhoods as shown in Figure~\hyperref[fig:gamma]{\ref{fig:gamma}(a)}.  By the lemma, there exists an algebraic curve $\Sigma \subset \cc^2$ and constants $0<r<R$ such that $\beta'=\Sigma \cap \partial B^4_r$ and $\gamma=\Sigma \cap \partial B^4_R$.

Next we show that $\gamma$ is an unknot. Through the smooth isotopy depicted in Figure~\hyperref[fig:gamma]{\ref{fig:gamma}(b-c)}, we see that $\gamma \in B_{n+m}$ is isotopic to a quasipositive braid $\gamma' \in B_n$ whose ribbon-immersed Bennequin surface in fact contains no singularities and is therefore embedded. Observe that this braid $\gamma'$ also has the same number of bands as $\beta$.  Since the slice-Bennequin inequality is sharp for quasipositive braids (\cite{rudolph:qp-obstruction,hedden:pos}), the number of bands in $\beta$ equals $n+2g_*(\beta)-1=n-1$. It follows that the Benequin surface for $\gamma'$ has Euler characteristic $-1$, implying that $\gamma'$ (and thus $\gamma$) is an unknot. 
 
Finally, we show that the compact piece of $\Sigma$ bounded by $\gamma$ is a holomorphically unknotted disk. For notational convenience, we rescale $\cc^2$ so that $R=1$ and let $\Sigma_1$ denote the resulting holomorphic curve in the unit four-ball.  That $\Sigma_1$ is a disk follows from the fact that holomorphic curves are genus-minimizing \cite{km:thom,rudolph:qp-obstruction}. To show that $\Sigma_1$ is holomorphically unknotted, we employ Eliashberg's technique of filling by holomorphic disks (as in \cite[Proposition~2]{boileau-fourrier}). The braid $\gamma$ can be viewed as a transverse unknot (in the contact-geometric sense) with self-linking number $-1$. It follows that it bounds an embedded disk $D\subset S^3$ whose characteristic foliation (see \cite{etnyre:intro})  is radial with a single positive elliptic point $p \in D$. There is a small interval $[0,\epsilon)$ and a family of disjoint holomorphic ``Bishop disks'' $\Delta_t \subset B^4$ for $t \in [0,\epsilon)$  emerging from $p$ such that $\Delta_t$ is properly embedded in $B^4$ with boundary in $D\subset S^3$ for $t >0 $ (and $\Delta_0$ is a degenerate disk identified with the given elliptic point). By \cite[Corollary~2.2B]{yasha:lagr-cyl}, the Bishop family can be extended to a family of holomorphic disks foliating a three-ball bounded by the piecewise-smooth two-sphere $D \cup \Sigma_1$. In particular, the extended family includes $\Sigma_1$, so $\Sigma_1$ can be unknotted through holomorphic disks.
\end{proof}

\begin{example}
The quasipositive slice knot $m(8_{20})$ is represented by the 3-braid $\beta=\sigma_2 \sigma_1^2\sigma_2\sigma_1^{-2} \sigma_2^{-1}\sigma_{1,3}$, which embeds in the 5-braid $\beta'=w \sigma_2 w^{-1} \sigma_{1,5}$, where $w= \sigma_4^{-1} \sigma_3^2 \sigma_4^{-2} \sigma_3^{-1} \sigma_2^{-1} \sigma_1^2 \sigma_2^{-2} \sigma_1^{-1} \sigma_1 \sigma_2$. Following the above procedure, $\beta'$ arises as a cross-section of a holomorphic disk in the four-ball bounded by the unknot $\gamma= w' \sigma_2 w^{-1} \sigma_{1,5}$, where $w'$ is obtained from $w$ by the two replacements $\sigma_3^2 \to \sigma_3 \sigma_4 \sigma_3$ and $\sigma_1^2 \to \sigma_1 \sigma_2 \sigma_1$.
\end{example}

\begin{remark}
The construction from the proof above generalizes to show that every quasipositive knot $K$ embeds as a link component in a cross-section of an algebraic curve in the four-ball whose boundary is a strongly quasipositive knot with the same slice genus as $K$.
\end{remark}

Next, we produce the algebraic curves needed for the refined solution to Problem~\hyperref[prob:kirby]{1.100(A)}. By \cite[Theorem~1.6]{hs:surgery} (see also \cite{st:genus}), band surgery decreases the Seifert genus of a knot if and only if the knot has a minimal genus Seifert surface that contains the band. Therefore, to violate the inequality from Problem~\hyperref[prob:kirby]{1.100(A)}, we seek   pairs of quasipositive braids $\beta$ and $\gamma$ such that $\gamma$ is obtained from $\beta$ by band surgery along a band lying inside a minimal genus Seifert surface for $\beta$.

\begin{proof}[Proof of Theorem~\ref{thm:100a}]
For each $n \geq 0$, define a pair of four-braids as follows:
\begin{align*}
\beta_n &=\sigma_{1,4} \sigma_2 (\sigma_1^n \sigma_{1,3}^{-1} \sigma_1^{-n}) \sigma_{2,4} (\sigma_1^n \sigma_{1,3} \sigma_1^{-n}) 
\\
\gamma_n &= \sigma_{1,4} \sigma_2  \sigma_{2,4} \sigma_1^{n} \sigma_{1,3} \sigma_2^{n+1}.
\end{align*}
Observe that each $\beta_n$ is quasipositive and each $\gamma_n$ is strongly quasipositive. Bennequin surfaces for $\beta_0$ and $\beta_1$ appear in Figure~\ref{fig:bennequin}. We can obtain $\gamma_n$ from $\beta_n$ by adding $2n+2$ positive bands: Adding $\sigma_{1,3}$ in the middle of the term $(\sigma_1^n \sigma_{1,3}^{-1} \sigma_1^{-n})$ in $\beta_n$ eliminates that entire term, adding $\sigma_1^n$ to the end of $\beta_n$ eliminates the term $\sigma_1^{-n}$, and then $\sigma_2^{n+1}$ is added to the end of the word. (The final term is added to ensure that $\gamma_n$ is a knot and to simplify the expression of its Seifert genus.) By the lemma, there exists an algebraic curve $\Sigma_n$ and constants $0<r <R$ such that $\beta_n= \Sigma_n \cap \partial B^4_r$ and $\gamma_n = \Sigma_n \cap \partial B^4_R$.

It remains to calculate the Seifert genera of $\beta_n$ and $\gamma_n$. Since $\gamma_n$ is strongly quasipositive, its Bennequin surface realizes the Seifert genus of $\gamma_n$; see \cite{rudolph:qp-obstruction}. This surface is built from 4 disks and $5+2n$ bands and has connected boundary, hence has 
genus 
$n+1$.

To compute the Seifert genus of $\beta_n$, we show that its Bennequin surface is also genus-minimizing. The Bennequin surface for $\beta_n$ is built from 4 disks and $5+4n$ bands, hence has Euler characteristic $\chi=-1-4n$. Since its boundary is a knot, the genus is $(1-\chi)/2=2n+1$. For $n=0$ and $n=1$, it is easy to confirm that $g(\beta_n)=2n+1$, and this was already completed for $n=0$ and $n=1$ in Example~\ref{ex:bennequin}.  We proceed by induction: Let us assume the Bennequin surface for $\beta_{n-1}$ is minimal, where $n \geq 2$. The braid $\beta_n$ can be viewed as being obtained from $\beta_{n-1}$ by four replacements of the form $\sigma_1^{\pm 1} \to \sigma_1^{\pm 2}$. As illustrated in Figure~\ref{fig:plumbing}, the replacement $\sigma_1 \to \sigma_1^2$ can be viewed as the result of plumbing a positive Hopf band across the $\sigma_1$-band. The replacement $\sigma_1^{-1} \to \sigma_1^{-2}$ is analogous, realized by plumbing a negative Hopf band across a $\sigma_1^{-1}$-band. It follows that the Bennequin surface for $\beta_n$ is obtained from that of $\beta_{n-1}$ by repeated Hopf plumbing. Since the Bennequin surface for $\beta_{n-1}$ is genus-minimizing, it follows from \cite{gabai:murasugi} that the Bennequin surface for $\beta_n$ is genus-minimizing as well.
\end{proof}

\begin{figure}\center
\def\svgwidth{.9\linewidth} 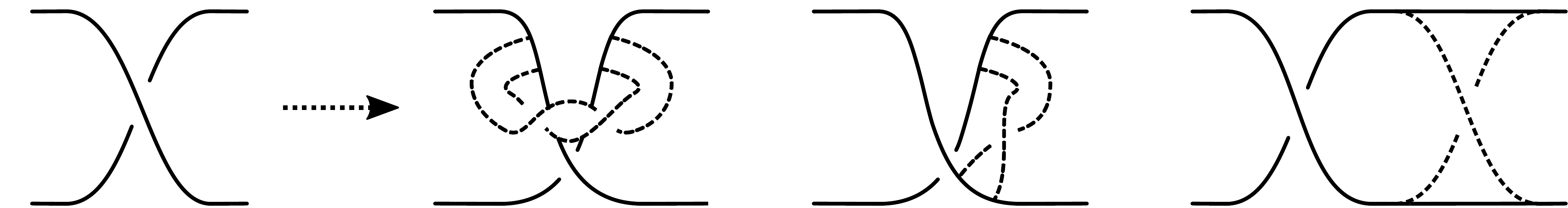 \caption{Hopf plumbing across a band of like sign.} \label{fig:plumbing}
\end{figure}

\section{Unknotted Lagrangian disks with nontrivial cross-sections}\label{sec:lagr}

In this final section, we adapt the construction from \S\ref{sec:top} to the Lagrangian setting. We assume the reader is familiar with the basics of Legendrian knot theory and front diagrams for Legendrian links; see, for example, Etnyre's survey \cite{etnyre:knot-intro}. For background on Lagrangian fillings of Legendrian links, see \cite{chantraine}.

We use the following theorem to construct the desired Lagrangian surfaces:

\begin{theorem}[\cite{bst:construct, chantraine, ehk:cobordisms, rizell:surgery}]
\label{thm:construct}
If two Legendrian links $L_{\pm}$ in the standard contact $S^3$ are related by a sequence of any of the following three moves, then there exists an exact, embedded, orientable, and collared Lagrangian cobordism from $L_-$ to $L_+$.
\begin{description}
\item[Isotopy] $L_{-}$ and $L_{+}$ are Legendrian isotopic.
\item[$\boldsymbol{0}$-Handle] The front of $L_{+}$ is the same as that of $L_{-}$ except for the addition of a disjoint Legendrian unknot as in the left side of Figure~\ref{fig:handle}.
\item[$\boldsymbol{1}$-Handle] The fronts of $L_{\pm}$ are related as in the right side of Figure~\ref{fig:handle}.
\end{description}
\end{theorem}

\begin{figure}[b]
\includegraphics[width=.5\linewidth]{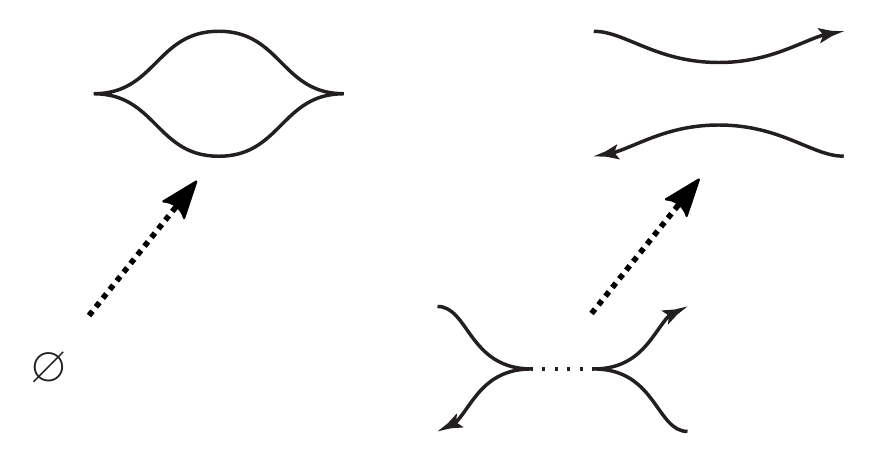} 
\caption{Diagram moves corresponding to attaching a 0-handle and an oriented 1-handle.}
\label{fig:handle}
\end{figure}

\begin{remark} A cross-section $\Sigma \cap S^3_r$ of a Lagrangian surface $\Sigma \subset B^4$ is \textbf{collared} if the radial vector field $\partial_r$ is tangent to $\Sigma$ near the cross-section $\Sigma \cap S^3_r$. More generally, a  cross-section $\Sigma \cap S^3_r$ is \textbf{collarable} if there exists $\epsilon>0$ and a family of isotopic Lagrangian surfaces $\Sigma_t$ with $\Sigma_0=\Sigma$ such that $\Sigma_t$ is transverse to $S^3_r$ for all $t \in [0,1]$ and $\Sigma_1 \cap S^3_r$ is collared.
\end{remark}

\begin{example}\label{ex:nontrivial}
Consider the sequences of front diagrams in Figure~\ref{fig:lagr}. The implied first step in each sequence is the attachment of some number of 0-handles. Applying Theorem~\ref{thm:construct}, part (a) corresponds to a Lagrangian filling of the knot $m(9_{46})$, which is seen to be a disk because it consists of two 0-handles joined by a single 1-handle. (This sequence is adapted from \cite[Example~4.1]{positivity}.) Part (b) corresponds to a Lagrangian disk bounded by a Legendrian unknot. Observe that the penultimate  front  in this sequence depicts a two-component link formed from the knot $11n_{139}$ and an unknot. 
\end{example}

\begin{figure}
\def\svgwidth{.9\linewidth} 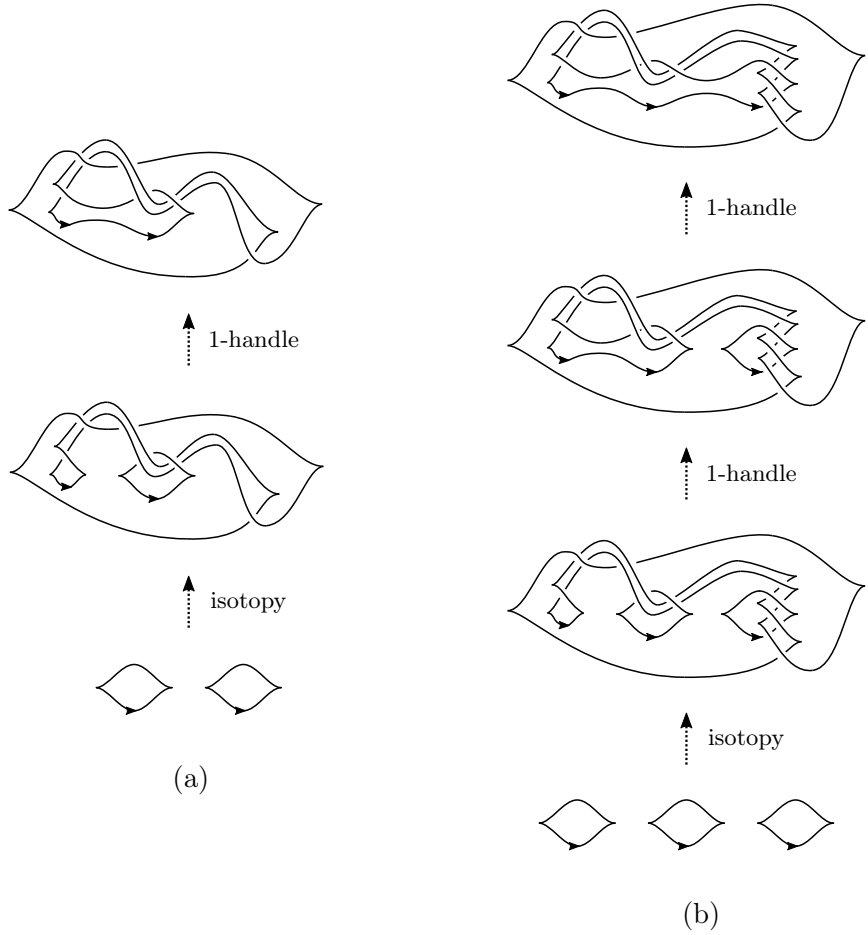 
\caption{Diagram sequences corresponding to Lagrangian disks.}
\label{fig:lagr}
\end{figure}

\begin{figure}
\def\svgwidth{\linewidth} 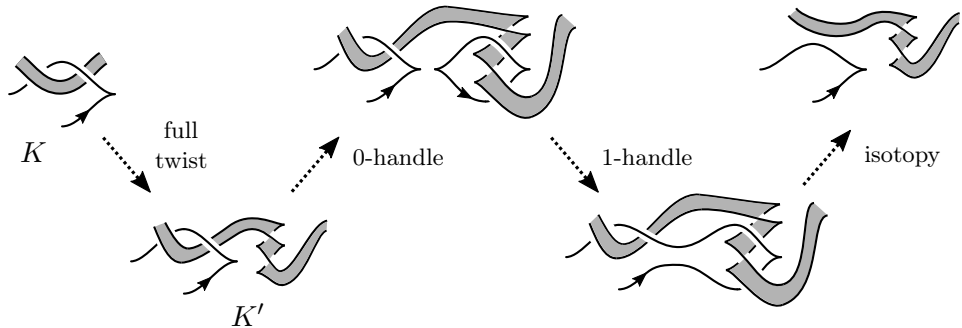 
\caption{Local modification to remove a ribbon singularity.}
\label{fig:replacement}
\end{figure}

\begin{proof}[Proof of Theorem~\ref{thm:lagr}] We begin by pointing out that the sequence in part (b) of Figure~\ref{fig:lagr} is obtained from the one in part (a) by applying the local modification depicted in Figure~\ref{fig:replacement}. We will prove the theorem by generalizing this construction.

Using Theorem~\ref{thm:construct}, it is simple to produce infinitely many Legendrian knots $K$ that are filled by Lagrangian disks. Moreover, we can produce infinite families such that the (necessarily ribbon) Lagrangian fillings give rise to ribbon-immersed surfaces in $S^3$ whose ribbon singularities have the local form depicted in the top left of Figure~\ref{fig:replacement}.  Let $K'$ denote the knot obtained from $K$ by adding a full twist to the band as shown. By construction, $K'$ also bounds a Lagrangian disk built using Theorem~\ref{thm:construct}. By introducing a 0-handle, we can embed $K'$ in  a two-component link as shown in the third stage of Figure~\ref{fig:replacement}. Adding a 1-handle as shown produces a knot that bounds a ribbon-immersed surface with the same number of 0- and 1-handles as the ribbon-immersed surface bounded by $K$, but which can be isotoped to have one fewer ribbon singularity.  Applying this to all ribbon singularities, we obtain a knot $K'$ as above and a Lagrangian filling of the unknot that contains $K'$ as a link component in a collared cross-section. For an appropriate choice of $K$, it is easy to ensure that $K'$ (and thus the cross-section containing it) is nontrivial.

To see that these Lagrangian fillings of the unknot are unknotted as embedded disks in $B^4$, we appeal to work of Eliashberg and Polterovich \cite{ep:local-knots} which implies that any two properly embedded Lagrangian disks in $B^4$ with collared, unknotted boundary are in fact isotopic through such Lagrangian disks. Therefore, the fillings of the unknot constructed above are isotopic to the standard unknotted filling of the unknot.
\end{proof}

\bibliographystyle{alpha}
\bibliography{biblio}

\end{document}